\documentclass[11pt]{article}
\usepackage[margin=2.4cm]{geometry}

\usepackage{graphicx}
\graphicspath{{./images}}
\usepackage{amsmath,amsthm,latexsym,amsfonts,amssymb,mathrsfs}
\usepackage[usenames]{color}
\usepackage{tikz}
\usetikzlibrary{math}

\usepackage{cancel,soul,ulem}

\usepackage{url}

\newcommand{\calB}{\mathcal{B}}
\newcommand{\calBhat}{\hat{\mathcal{B}}}
\newcommand{\calL}{\mathcal{L}}

\newcommand{\vf}{{\bf f}}
\newcommand{\vn}{{\bf n}}
\newcommand{\vu}{{\bf u}}
\newcommand{\vv}{{\bf v}}

\newcommand{\vV}{{\bf V}}
\newcommand{\vH}{{\bf H}}
\newcommand{\vVhat}{\mathbf{\hat V}}
\newcommand{\vHhat}{\mathbf{\hat H}}
\newcommand{\vw}{{\bf w}}

\newcommand{\bsx}{{\boldsymbol{x}}}

\newcommand{\vsigma}{\boldsymbol{\sigma}}

\newcommand{\R}{\mathbb{R}}

\numberwithin{equation}{section}

\newcommand{\veps}{{\boldsymbol{\varepsilon}}}

\newcommand{\ud}{\mathrm{d}}
\newcommand{\diam}{\mathrm{diam}}
\newtheorem{theorem}{Theorem}[section]
\newtheorem{lemma}[theorem]{Lemma}
\newtheorem{corollary}[theorem]{Corollary}

\newtheorem{assumption}[theorem]{Assumption}
\newtheorem{remark}[theorem]{Remark}

\title{A simple modification to mitigate locking in conforming FEM\\ for nearly incompressible elasticity
 \thanks{School of Mathematics and Statistics, University of New South Wales, Sydney, Australia}\thanks{This work was supported by the Australian Research Council
grant DP220101811.}}
\author{K. Mustapha, W. McLean, J. Dick and  Q. T. Le Gia}
\begin{document}
\maketitle
\begin{abstract}
Due to the divergence-instability, the accuracy of low-order conforming finite element methods (FEMs)  for nearly incompressible elasticity equations deteriorates as the Lam\'e parameter $\lambda\to\infty$, or equivalently as the Poisson ratio $\nu\to1/2$.  This effect is known as {\itshape locking} or {\itshape non-robustness}. For the piecewise linear case, the error in the ${\bf L}^2$-norm of the standard Galerkin conforming FEM is bounded by~$C\lambda h^2$, resulting in poor accuracy for practical values of~$h$ if $\lambda$ is sufficiently large. In this short paper, we show that the locking phenomenon can be reduced by replacing $\lambda$ with~$\lambda_h=\lambda\mu/(\mu+\lambda h/L)<\lambda$ in the stiffness matrix, where $\mu$ is the second Lam\'e parameter and $L$ is the diameter of the body $\Omega$. We prove that with this modification, the error in the ${\bf L}^2$-norm is bounded by $Ch$ for a constant $C$ that does not depend on $\lambda$.  Numerical experiments confirm this convergence behaviour and show that, for practical meshes, our  method is more accurate than the standard method if $\lambda$ is larger than about $\mu L/h$.  Our analysis also shows that the error in the ${\bf H}^1$-norm is  bounded by $C\lambda_h^{1/2}\,h$, which improves the  $C\lambda^{1/2}\,h$ estimate for the case of conforming FEM.
\end{abstract}
\section{Introduction}
Let $\Omega\subseteq\R^d$ denote a polygon if $d=2$, or a polyhedron if 
$d=3$, and consider the following linear elasticity equation:  
\begin{equation}
 -\nabla \cdot \vsigma(\vu(\bsx)) = \vf(\bsx) \quad \text{for $\bsx \in \Omega$,}
\label{eq:L1}
\end{equation}
subject to a homogeneous Dirichlet condition $\vu = {\bf 0}$ on the
boundary $\Gamma:=\partial \Omega$. The Cauchy stress
tensor $\vsigma \in [L^2(\Omega)]^{d\times d}$ is defined as 
\[\vsigma(\vu) = \lambda \nabla\cdot \vu\, {\bf I} 
+ 2\mu \veps(\vu)\quad\text{on $\Omega$,} \] 
with $\vu$ being the displacement vector field and the symmetric strain 
tensor $\veps(\vu) := \frac{1}{2} (\nabla \vu + (\nabla \vu)^T)$. Here, $\vf$ is
the body force per unit volume, and ${\bf I}$ is the identity tensor. The
gradient ($\nabla$) and the divergence ($\nabla \cdot$) are understood to be
with respect to the physical variable $\bsx \in \Omega$.  The constant Lam\'e
parameters are assumed to satisfy $(\lambda,\mu)\in[\lambda_0,\infty)\times[\mu_0,\mu_1]$ 
with $\lambda_0>0$ and $0<\mu_0<\mu_1<\infty$. We are interested
in the case when the Poisson ratio $\nu$ of the elastic material
approaches $1/2$ or, equivalently, $\lambda$ approaches infinity, and thus the material is
nearly incompressible.  It will frequently 
prove convenient to work with the dimensionless ratio~$\alpha=\lambda/\mu\gg1$.

The convergence rate may deteriorate (as it depends on $\lambda$) when conforming
FEMs are employed to approximate the solution of~\eqref{eq:L1}. This issue is
commonly referred to as locking. It has been shown that locking is absent for
polynomials of degree $\ge 4$ on a triangular mesh~\cite{ScottVogelius1985}.
However, on quadrilateral meshes, locking cannot be  avoided for any polynomial
of degree $\ge1 $ \cite{BabuskaSuri1992}.  In a related result,
Vogelius~\cite{Vogelius1983} proved that the $p$-version of the conforming FEM
on smooth domains is  unaffected by locking. The survey paper of Ainsworth and
Parker~\cite{AinsworthParker2022} illustrates and analyses some of the
surprising effects that can occur when $\lambda$ is very large.

The locking phenomenon in conforming FEMs has motivated researchers to explore
and develop various  locking-free numerical methods ($h$-, $p$-, and
$hp$-versions, low-order, and high-order) for solving the nearly incompressible elasticity equation \eqref{eq:L1}. These methods include  nonconforming FEMs
\cite{ArnoldAwanouWinther2014,BrennerSung1992,ChenRenMao2010,Falk1991,Lamichhane2015,LeeLeeSheen2003,MaoChen2008,YangChen2010}, conforming-nonconforming mixed FEMs \cite{BoffiBrezziFortin2013,
GopalakrishnanGuzman2011,HuShi2008,ZhangZhaoChenYang2018}, the weak Galerkin FEM \cite{ChenXie2016,
HuoWangWangZhang2020,HuoWangWangZhang2023,LiuWang2022,WangWangWangZhang2016},
the discontinuous Petrov--Galerkin and Galerkin methods  \cite{BramwellDemkowiczGopalakrishnanQiu2012, CockburnSchotzauWang2006,DiPietroNicaise2013,
HansboLarson2002,Wihler2006}, and the virtual element method~\cite{EdoardoStefanoCarloLuca2020,HuangLinYu2022,VeigaBrezziMarini2013,
ZhangZhaoYangChen2019,ZhaoWangZhang2022}.  It is worth mentioning that most of these methods are more complicated, both computationally and in their analysis,
than the conforming FEM. Furthermore, these methods become more complicated  for the case of an inhomogeneous material, where the Lam\'e parameters are functions of $\bsx$.

{Another approach that might be used to avoid locking in the
conforming linear FEM is the so-called reduced integration technique
\cite{HughesCohenHaroun1978,MalkusHughes1978}, in which a finite element
interpolation operator $\Pi$ has the property $\nabla\cdot(\Pi \vv)=0$ in
a weak sense whenever $\nabla\cdot\vv=0$.  The same procedure was used
for the case of a pure traction problem~\cite{BrennerSung1992}. 

The aim of the present work is to demonstrate that the locking phenomenon in
the popular piecewise linear conforming FEM can be mitigated simply by replacing
the Lam\'e parameter $\lambda$ with $\lambda_h=\lambda\mu/(\mu+\lambda h/L)$ in the
stiffness matrix, where $L$ is the diameter of~$\Omega$. Although 
our analysis is independent of the dimension~$d$, we rely on some delicate technical 
estimates referred to below as Assumption~\ref{assumption}.   These estimates are known to 
hold if $d=2$ and if the polygon~$\Omega$ is convex.

In the next section, we succinctly discuss the weak formulation of
Problem~\eqref{eq:L1} and the well-posedness of the modified conforming FEM, as
well as setting up a formalism to keep track of how terms in our estimates
scale with~$L$. In Section~\ref{ConvAna}, we demonstrate that the
scheme is $O(h) $ accurate, uniformly in $\lambda$. Additionally, we include a remark on the applicability of these results in the case of mixed Dirichlet-traction boundary conditions. Furthermore, we include another remark on how to extend the proposed finite element scheme and the convergence analysis for the case of an inhomogeneous body, where  the Lam\'e parameters $\lambda$ and $\mu$ depend on the space variable $\bsx \in \Omega$.  
To support our theoretical findings, we present some numerical results in Section~\ref{Sec: Numeric} for two different examples,  including Cook's benchmark problem.
\section{Modified conforming FEM}\label{Sec: FEM}
In this section, we discuss the existence and uniqueness of the
weak solution of \eqref{eq:L1},   define our piecewise linear, modified conforming 
Galerkin FEM, and then introduce some scale-invariant norms before discussing the required 
regularity properties for the convergence analysis. We investigate the error estimates in 
the next section. Throughout the paper, $C$~and $c$ denote generic 
(large and small) positive constants, independent of $h$, $L$, $\lambda$ and hence also 
$\alpha=\lambda/\mu$.

In the weak formulation of  equation~\eqref{eq:L1}, we seek $\vu \in \vV$
satisfying
\[
\int_\Omega[2\mu\,\veps(\vu):\veps(\vv)+\lambda(\nabla\cdot\vu)(\nabla\cdot\vv)]\,\ud\bsx
= \int_\Omega \vf \cdot \vv \,\ud\bsx 
\quad \text{for all~ $\vv \in \vV$,}
\]
where the colon operator is the inner product between tensors.  
We define the linear functional $\ell$ and the bilinear form $\calB$ by
\begin{equation*}\label{eq: bilinear}
\ell(\vv) = \int_\Omega\mu^{-1}\vf \cdot \vv \,\ud\bsx 
 \quad\text{and}\quad
\calB(\vu,\vv)=\int_\Omega[2\veps(\vu):\veps(\vv)
    +\alpha(\nabla\cdot\vu)(\nabla\cdot\vv)]\,\ud\bsx,
\end{equation*}
so that, after dividing by~$\mu$, the weak formulation can be expressed as
\begin{equation}\label{para weak}
 \calB(\vu, \vv) = \ell(\vv), \quad \text{for all~ $\vv \in \vV$.}
\end{equation}
Let $\|\cdot\|$ denote the norm in ${\bf L}^2(\Omega)=[L^2(\Omega)]^d$, and let
$\|\cdot\|_{\vV} $ and $\|\cdot\|_{\vH}$ denote norms in
$\vV=[H^1_0(\Omega)]^d $ and ${\vH}=[H^2(\Omega)]^d$, respectively.  (Below, we
will make specific choices of the latter two norms.) Here, $H^1_0(\Omega) $ and
$H^2(\Omega)$ are the usual Sobolev spaces of scalar-valued functions on $\Omega$.  Finally, $\vV^*$ is 
the dual space of $\vV$ with respect to the inner product in ${\bf L}^2(\Omega)$. 
 
Since $\calB(\vv,\vw) 
\le C\alpha\| \vv\|_{\vV}\,\|\vw\|_{\vV}$ for $\vv,\vw \in \vV$, the bilinear 
form $\calB(\cdot,\cdot)$ is bounded over $\vV \times \vV$. The coercivity of $\calB(\cdot,\cdot)$ 
on $\vV$ follows from Korn's inequality,
\begin{equation}\label{eq: coer B}
\|\vv\|_\calB^2:=\calB(\vv,\vv)\ge2\|\veps(\vv)\|^2\ge c\|\vv\|_{\vV}^2,
\quad\text{for all $\vv\in\vV$.}
\end{equation}
Owing to these two properties, and since $|\ell(\vv)|
\le \|\mu^{-1}\vf\|_{\vV^*}\|\vv\|_{\vV}$, an application of
the Lax--Milgram theorem completes the proof of the next theorem. 

\begin{theorem}\label{thm: unique solution}
For every $f\in \vV^*$, the problem~\eqref{para weak} has a unique
solution $\vu\in\vV$.
\end{theorem}

For the finite element scheme, we introduce as usual a family of regular 
quasi-uniform triangulations $\mathcal{T}_h$ of the domain $\Omega$ and let 
$h=\max_{K\in \mathcal{T}_h}(h_K)$, where $h_{K}$ denotes the diameter of the element $K$.  The space of polynomials with degree at most $1$ is denoted
by $P_1$, and the conforming finite element space is then defined by
\[\vV_h=\{\,\vw_h \in {\bf V}:\text{$\vw_h |_K \in [P_1]^2$ for all $K \in \mathcal{T}_h$}\,\}.
\]
In the standard conforming FEM we seek $\vu_h \in  \vV_h$ such that
$\calB(\vu_h, \vv_h) = \ell(\vv_h)$ for all $\vv_h \in  \vV_h$. To
control locking when $\lambda$ is large, more precisely,
when $\lambda$ is larger than about $L/h$, we modify this scheme as
follows: find $\vu_h \in  \vV_h$ such that
\begin{equation}\label{FEM new}
\calB_h(\vu_h,\vv_h)=\ell(\vv_h),\quad\text{for all $\vv_h\in\vV_h$}.
\end{equation}
The bilinear operator $\calB_h$ is defined in the same way as~$\calB$
in~\eqref{eq: bilinear}, except with~$\alpha_h$ in place of~$\alpha$,
that is,
\begin{equation*}
\calB_h(\vu,\vv)=\int_\Omega[2\veps(\vu):\veps(\vv)
  +\alpha_h(\nabla\cdot\vu)(\nabla\cdot\vv)]\,\ud\bsx,
\end{equation*} 
for a dimensionless parameter~$\alpha_h$ given by~\eqref{eq: alpha_h} below.   Equivalently, we replace 
$\lambda$ with~$\lambda_h=\alpha_h\mu$ in the standard finite element formulation. The existence and
uniqueness of the solution~$\vu_h$ of~\eqref{FEM new} is guaranteed because the
operator~$\calB_h$ is bounded and positive-definite on~$\vV_h$.  For future
reference, we note that
\begin{equation}\label{eq: B-B_h}
\calB(\vw,\vv)-\calB_h(\vw,\vv)
  =\int_\Omega (\alpha-\alpha_h)(\nabla\cdot\vw)(\nabla\cdot\vv)\,\ud\bsx
  \quad\text{for all $\vw$, $\vv\in\vV$.}
\end{equation}
The choice of $\alpha_h$ will follow by balancing two terms that scale
differently with the size of the domain~$\Omega$. We therefore introduce a reference domain with 
unit diameter,
\[
\widehat\Omega=L^{-1}\Omega=\{\,L^{-1}\bsx:\bsx\in\Omega\,\}
\quad\text{where $L=\diam(\Omega)$.}
\]
To each function $\vv:\Omega\to\mathbb{R}^d$ we associate a
function $\hat\vv:\widehat\Omega\to\mathbb{R}^d$ defined by
\[
\hat\vv(\hat\bsx)=\vv(\bsx)
  \quad\text{for $\bsx=L\hat\bsx\in\Omega$ and $\hat\bsx\in\widehat\Omega$.}
\]
For $r\ge 1,$ denote the Sobolev seminorm by $|\vv|_{r,\Omega}=\bigl(\sum_{|\beta|=r}
\|\partial^\beta\vv\|^2\bigr)^{1/2}$, and observe that
\[
\|\vv\|=L^{d/2}\|\hat\vv\|_{\mathbf{L}^2(\widehat\Omega)}
\quad\text{and}\quad
|\vv|_{r,\Omega}=L^{d/2-r}|\hat\vv|_{r,\widehat\Omega}
\]
because $\ud\bsx=L^d\ud\hat\bsx$ and
$\partial^\beta\vv(\bsx)=L^{-|\beta|}\partial^\beta\hat\vv(\hat\bsx)$.
Equip $\vV$ and $\vH$ with the specific norms given by
\begin{equation}\label{eq: V H norms}
\|\vv\|_{\vV}^2=L^{-d}\|\vv\|^2+L^{2-d}|\vv|_{1,\Omega}^2
\quad\text{and}\quad
\|\vv\|_{\vH}^2=\|\vv\|_{\vV}^2+L^{4-d}|\vv|_{2,\Omega}^2,
\end{equation}
and the corresponding spaces
$\vVhat=[H^1_0(\widehat\Omega)]^2$ and $\vHhat=[H^2(\widehat\Omega)]^2$ with norms given by
\[\|\hat\vv\|_{\vVhat}^2=\|\hat\vv\|_{\mathbf{L}^2(\widehat\Omega)}^2
  +|\hat\vv|_{1,\widehat\Omega}^2
\quad\text{and}\quad
\|\hat\vv\|_{\vHhat}^2=\|\hat\vv\|_{\vVhat}^2
  +|\hat\vv|_{2,\widehat\Omega}^2.
\]
In this way, $\|\vv\|_{\vV}=\|\hat\vv\|_{\vVhat}$ and 
$\|\vv\|_{\mathbf{H}}=\|\hat\vv\|_{\vHhat}$.

Since $\veps(\vv)=L^{-1}\veps(\hat\vv)$ and
$\nabla\cdot\vv=L^{-1}\nabla\cdot\hat\vv$, if $\vu\in\vV$ satisfies
\eqref{para weak}, then $\hat\vu\in\vVhat$ will
satisfy
\begin{equation}\label{eq: u hat solution}
\calBhat(\hat\vu,\hat\vv)=\hat\ell(\hat\vv)
\quad\text{for all $\hat\vv\in\vVhat$,}
\end{equation}
where
\begin{equation}\label{eq: ell B hat}
\hat\ell(\hat\vv) = \int_{\widehat\Omega}\mu^{-1}L^2\hat\vf\cdot\hat\vv\,\ud\hat\bsx
\quad\text{and}\quad
\calBhat(\hat\vu,\hat\vv)=\int_{\widehat\Omega}
[2\veps(\hat\vu):\veps(\hat\vv)+\alpha(\nabla\cdot\hat\vu)
  (\nabla\cdot\hat\vv)]\,\ud\hat\bsx.
\end{equation}
Noting that $\hat\ell(\hat\vv)=L^{2-d}\ell(\vv)$ and
$\|\hat\vv\|_{\vVhat}=\|\vv\|_{\vV}$, with $\|\vv\|\le L^{d/2}\|\vv\|_{\vV}$,
we see that
\begin{equation}\label{eq: norm ell}
\|\hat\ell\|_{\mathbf{L}^2(\widehat\Omega)}=L^{2-d/2}\|\ell\|
=\mu^{-1}L^{2-d/2}\|f\|~\text{and}~
\|\hat\ell\|_{\vVhat^*}=L^{2-d}\|\ell\|_{\vV^*}\le L^{2-d/2}\|\ell\|
= \mu^{-1}L^{2-d/2}\|\vf\|.
\end{equation}

The map $\bsx\mapsto\hat\bsx$ takes the triangulation of $\Omega$ to a
triangulation of $\widehat\Omega$ with mesh size $\hat h=h/L$, and we have
the corresponding finite element space
$\vVhat_h=\{\,\hat\vv_h:\vv_h\in\vV_h\,\}$, with $\hat\vu_h\in\vVhat_h$
satisfying
\[
\calBhat_h(\hat\vu_h,\hat\vv_h)=\hat\ell(\hat\vv_h)
\quad\text{for $\hat\vv_h\in\vVhat_h$,}
\]
where $\calBhat_h$ is defined in the same way as~$\calBhat$
in~\eqref{eq: ell B hat} except with~$\alpha_h$ in place of~$\alpha$.
Our convergence analysis relies on the following regularity estimates, in
which $C$ depends only on~$\widehat\Omega$.

\begin{assumption}\label{assumption}
The weak solution~$\hat\vu$ from~\eqref{eq: u hat solution} satisfies
\[
\|\hat\vu\|_{\hat\vV}+\alpha\|\nabla\cdot\hat\vu\|_{L^2(\widehat\Omega)}
    \le C\|\hat\ell\|_{\hat\vV^*}
\quad\text{and}\quad
\|\hat\vu\|_{\hat\vH}\le C\|\hat\ell\|_{L^2(\widehat\Omega)}.
\]
\end{assumption}

This assumption is known to hold when $d=2$~and $\Omega$ is a convex polygon, thanks to
results by Brenner and Sung~\cite[Lemma~2.2]{BrennerSung1992} and by Bacuta and
Bramble~\cite[Equation~(3.2)]{BacutaBramble2003}.

\begin{theorem}\label{thm: vu bound}
If Assumption~\ref{assumption} is satisfied, then the weak solution of~\eqref{para weak}
satisfies
\begin{equation}\label{a priori}
\|\vu\|_{\vV}+\alpha L^{1-d/2}\|\nabla\cdot\vu\|\le C L^{2-d/2}\|\ell\|\qquad{\rm and}\qquad 
\|\vu\|_{\vH}\le CL^{2-d/2}\|\ell\|.
\end{equation}
\end{theorem}

\begin{proof}
Using Assumption~\ref{assumption} and \eqref{eq: norm ell}, we find that
\[\|\vu\|_{\vV}+\alpha L^{1-d/2} \|\nabla\cdot\vu\| =\|\hat\vu\|_{\hat\vV}+\alpha \|\nabla\cdot\hat\vu\|_{L^2(\widehat\Omega)}
\le C\|\hat\ell\|_{\hat\vV^*}
\le CL^{2-d/2}\|\ell\|.\] Furthermore, $\|\vu\|_{\vH}=\|\hat\vu\|_{\hat\vH} 
\le C\|\hat\ell\|_{L^2(\widehat\Omega)}=CL^{2-d/2}\|\ell\|$.
\end{proof}

\section{Convergence analysis}\label{ConvAna}
For the error analysis, we use a projection
operator $\pi_h:\vVhat\to\vVhat_h$ having the approximation property~\cite{CrouzeixRaviart1973}
\[
\|\hat\vv- \pi_h\hat\vv\|_{\mathbf{L}^2(\widehat\Omega)}
  +\hat h\|\nabla(\hat\vv- \pi_h\hat\vv)
  \|_{\mathbf{L}^2(\widehat\Omega)}
  \le C\hat h^2\|\hat\vv\|_{\vHhat},
\]
together with a corresponding projection $\Pi_h:\vV\to\vV_h$ satisfying
$\Pi_h\vv(\bsx)= \pi_h\hat \vv(\hat\bsx)$.  Using the scaling properties
described in Section~\ref{Sec: FEM} we arrive at
\begin{equation} \label{projection estimate}
\|\vv-\Pi_h\vv\|+h\|\nabla(\vv-\Pi_h\vv)\|
  \le CL^{d/2-2}h^2\|\vv\|_{\vH},
\end{equation}
with the same constant $C$.
The regularity results in Theorem~\ref{thm: vu bound} imply that
\begin{equation}\label{projection estimate Bh}
\begin{aligned}
\|\vu-\Pi_h \vu\|_{\calB_h}^2&=2\|\veps(\vu-\Pi_h \vu)\|^2
   +\alpha_h\|\nabla \cdot (\vu-\Pi_h \vu)\|^2\\
  &\le C(1+\alpha_h)(L^{d/2-2}h)^2\|\vu\|_{\vH}^2
   \le C \alpha_h\,h^2\|\ell\|^2,
\end{aligned}
\end{equation}
where $\|\vv\|_{\calB_h}=\sqrt{\calB_h(\vv,\vv)}$ is the norm associated with the 
operator $\calB_h$,  and we have implicitly assumed that $\alpha_h\ge c>0$ for some positive constant $c$.

The appropriate choice for~$\alpha_h$ will follow from the following preliminary error estimate.

\begin{lemma}\label{lem: balance terms}
$\|\vu_h-\Pi_h \vu\|_{\calB_h}^2\le C\bigl(\alpha_h\,h^2
  +(\alpha-\alpha_h)^2\alpha_h^{-1}\alpha^{-2}L^2\bigr)\|\ell\|^2$.
\end{lemma}
\begin{proof}
From the weak formulation in~\eqref{para weak} and the numerical scheme
in~\eqref{FEM new}, and using the identity~\eqref{eq: B-B_h} with $\vw=\vu $ and
$\vv=\vv_h$, we have the property
\begin{align}\label{equ:Gal ort}
\calB_h(\vu_h-\vu, \vv_h) &
  =\ell(\vv_h)-\calB_h(\vu,\vv_h)
  =\calB(\vu,\vv_h)-\calB_h(\vu,\vv_h)  \nonumber \\
  &=\int_\Omega(\alpha-\alpha_h)\nabla\cdot\vu\,\nabla\cdot\vv_h\,\ud\bsx
\quad\text{for all $\vv_h \in \vV_h$.}
\end{align}
Since $\|\vu_h-\Pi_h \vu\|_{\calB_h}^2=\calB_h(\vu-\Pi_h\vu,\vu_h-\Pi_h\vu)
+\calB_h(\vu_h-\vu,\vu_h-\Pi_h\vu)$, by applying the Cauchy--Schwarz inequality
and taking $\vv_h=\vu_h-\Pi_h\vu$ in~\eqref{equ:Gal ort}, we obtain
\begin{multline*}
\|\vu_h-\Pi_h \vu\|_{\calB_h}^2=\calB_h(\vu-\Pi_h\vu,\vu_h-\Pi_h\vu)
  +\int_\Omega(\alpha-\alpha_h)\,\nabla\cdot\vu\,\nabla\cdot
  (\vu_h-\Pi_h\vu)\,\ud\bsx\\
  \le\|\vu-\Pi_h\vu\|_{\calB_h}^2+\tfrac14\|\vu_h-\Pi_h \vu\|_{\calB_h}^2
  +(\alpha-\alpha_h)^2\alpha_h^{-1}\|\nabla\cdot\vu\|^2
  +\tfrac{1}{4} \|\alpha_h^{1/2}\,\nabla\cdot(\vu_h-\Pi_h\vu)\|^2.
\end{multline*}
Using $ \|\alpha_h^{1/2}\,\nabla\cdot(\vu_h-\Pi_h\vu)\|^2
\le\|\vu_h-\Pi_h\vu\|_{\calB_h}^2$ and then cancelling the similar terms gives
\[
\|\vu_h-\Pi_h\vu\|_{\calB_h}^2\le2\|\vu-\Pi_h\vu\|_{\calB_h}^2
  +2(\alpha-\alpha_h)^2\alpha_h^{-1}\|\nabla\cdot\vu\|^2.
\]
To complete the proof, combine \eqref{projection estimate Bh} with the regularity
results in Theorem~\ref{thm: vu bound}.
\end{proof}

We now choose  $\alpha_h$ so that 
$\alpha_hh^2=(\alpha-\alpha_h)^2\alpha_h^{-1}\alpha^{-2}L^2$.
Rearranging, $(h/L)^2=(\alpha_h^{-1}-\alpha^{-1})^2$, which leads to
\begin{equation}\label{eq: alpha_h}
\alpha_h=\frac{\alpha}{1+\alpha h/L}<\alpha
\quad\text{and hence}\quad
\lambda_h=\alpha_h\mu=\frac{\lambda\mu}{\mu+\lambda h/L}<\lambda.
\end{equation}

\begin{theorem}\label{Convergence theorem}
Let $\vu $ and $\vu_h$ be the solutions of problems
\eqref{eq:L1}~and \eqref{FEM new}, and define $\alpha_h$ by~\eqref{eq: alpha_h}.
If $\vf\in\mathbf{L}^2(\Omega)$ and Assumption~\ref{assumption} is satisfied, then
\begin{equation}\label{h1 error}
\|\vu_h-\vu\|_{\calB_h}\le  C\alpha_h^{1/2} h\|\ell\|.
\end{equation}
Moreover, if $\calL:{\bf L}^2(\Omega) \to \R$ is a bounded  linear functional,
so that $|\calL(\vw)|\le \|\calL\|\,\|\vw\|$, then
\begin{equation}\label{convergence calL}
|\calL(\vu_h)-\calL(\vu)|\le  C L h\|\ell\|\|\calL\|.
\end{equation}
\end{theorem}
\begin{proof} 
Our choice of $\alpha_h$ means that
$(\alpha-\alpha_h)^2\alpha_h^{-1}\alpha^{-2}L^2=\alpha_h h^2$, and therefore by
Lemma~\ref{lem: balance terms},
\begin{equation}\label{eq: uh Pih u}
\|\vu_h-\Pi_h \vu\|_{\calB_h}^2   \le  C\alpha_h h^2\|\ell\|^2.
\end{equation}
This inequality, together with  \eqref{projection estimate Bh}, shows \eqref{h1 error}. To prove the second estimate~\eqref{convergence calL}, we replace 
$\ell$ with $\calL$ in \eqref{para weak}, and then, by
Theorem~\ref{thm: unique solution}, there exists a unique $\vu_{\calL} \in \vV$ satisfying
\begin{equation}\label{weak form uL}
\calB(\vu_{\calL}, \vv) = \calL(\vv)\quad\text{for all $\vv \in \vV$.}
\end{equation}
Furthermore, using  Theorem \ref{thm: vu bound} (with $\vu_{\calL}$ in place of $\vu$), we deduce that  
\begin{equation}\label{regularity of u_L}
\|\vu_{\calL}\|_{\vV}+\alpha L^{1-d/2}\|\nabla\cdot \vu_{\calL}\|\le CL^{2-d/2}\|\calL\|\qquad{\rm and}\qquad 
\|\vu_{\calL}\|_{\bf H}\le CL^{2-d/2}\|\calL\|.
\end{equation}
The modified conforming finite element approximation to $\vu_{\calL}$ is the
function $\mathbf{\Theta}_h\in\vV_h$ satisfying
\begin{equation}\label{weak form uL discrete}
\calB_h(\mathbf{\Theta}_h,\vv_h) = \calL(\vv_h)
\quad\text{for all $\vv_h \in\vV_h$.}
\end{equation}
By repeating the arguments that led to~\eqref{h1 error}, and using
\eqref{regularity of u_L}, we obtain
\begin{equation}\label{projection estimate ucalL}
\|\mathbf{\Theta}_h-\vu_{\calL}\|_{\calB_h}\le C\alpha_h^{1/2}\, h \|\calL\|.
\end{equation}

By \eqref{FEM new} and \eqref{weak form uL discrete}, we have
$\calL(\vu_h)=\calB_h(\mathbf{\Theta}_h,\vu_h)
=\calB_h(\vu_h,\mathbf{\Theta}_h)=\ell(\mathbf{\Theta}_h)
=\calB(\vu,\mathbf{\Theta}_h)$, whereas \eqref{weak form uL}~and
\eqref{para weak} imply that
$\calL(\vu)=\calB(\vu_{\calL},\vu)=\calB(\vu,\vu_{\calL})$, so
\[
\calL(\vu_h)-\calL(\vu)=\calB(\vu,\mathbf{\Theta}_h-\vu_{\calL}).
\]
Thus, taking $\vw=\vu$ and $\vv=\mathbf{\Theta}_h-\vu_{\calL}$ in the
identity~\eqref{eq: B-B_h},
\begin{equation}\label{eq: functional difference}
\calL(\vu_h)-\calL(\vu)
  =\calB_h(\vu,\mathbf{\Theta}_h-\vu_{\calL})
  +\int_{\Omega}(\alpha-\alpha_h)(\nabla\cdot\vu)
    (\nabla \cdot(\mathbf{\Theta}_h-\vu_{\calL}))\,\ud\bsx.
\end{equation}
By \eqref{weak form uL}~and \eqref{weak form uL discrete} it follows that
\begin{align*}
\calB_h(\mathbf{\Theta}_h-\vu_{\calL},\vv_h)
  &=\calB_h(\mathbf{\Theta}_h,\vv_h)-\calB_h(\vu_{\calL},\vv_h)
  =\calL(\vv_h)-\calB_h(\vu_{\calL},\vv_h)
  =\calB(\vu_{\calL},\vv_h)-\calB_h(\vu_{\calL},\vv_h)\\
 &=\int_{\Omega}(\alpha-\alpha_h)
  (\nabla\cdot\vu_{\calL})(\nabla\cdot\vv_h)\,\ud\bsx
  \quad\text{for all $\vv_h \in\vV_h$.}
\end{align*}
In particular, choosing $\vv_h=\vu_h$, we see that
\begin{align*}
\calB_h(\vu,\mathbf{\Theta}_h-\vu_{\calL})
  &=\calB_h(\mathbf{\Theta}_h-\vu_{\calL},\vu-\vu_h)
  +\calB_h(\mathbf{\Theta}_h-\vu_{\calL},\vu_h)\\
  &=\calB_h(\vu-\vu_h,\mathbf{\Theta}_h-\vu_{\calL})
  +\int_{\Omega} (\alpha-\alpha_h)
  (\nabla\cdot\vu_{\calL})(\nabla\cdot\vu_h)\,\ud\bsx
\end{align*}
and hence by~\eqref{eq: functional difference},
\begin{align*}
\calL(\vu_h)-\calL(\vu)
  &=\calB_h(\vu-\vu_h,\mathbf{\Theta}_h-\vu_{\calL})
+\int_{\Omega}(\alpha-\alpha_h)
  (\nabla\cdot\vu)\nabla\cdot(\mathbf{\Theta}_h-\vu_{\calL})\,\ud\bsx\\
&\qquad{}+\int_{\Omega}(\alpha-\alpha_h)(\nabla\cdot\vu_{\calL})\bigl[
  \nabla\cdot(\vu_h-\vu)+\nabla\cdot\vu\bigr]\,\ud\bsx.
\end{align*}
Theorem~\ref{thm: vu bound}~and \eqref{regularity of u_L} imply that
$\alpha\|\nabla\cdot\vu\|\le CL\|\ell\|$ and
$\alpha\|\nabla\cdot\vu_{\calL}\|\le CL\|\calL\|$. Thus, from the achieved
estimates in \eqref{h1 error}~and \eqref{projection estimate ucalL}, and
using $\alpha_h^{1/2}\|\nabla\cdot\vv\|\le\|\vv\|_{\calB_h}$,
 \begin{align*}
|\calL(\vu)-\calL(\vu_h)|
  &\le\|\vu-\vu_h\|_{\calB_h}\|\mathbf{\Theta}_h-\vu_{\calL}\|_{\calB_h}
  +(\alpha-\alpha_h) \|\nabla\cdot\vu\|
    \|\nabla\cdot(\mathbf{\Theta}_h-\vu_{\calL})\|\\
  &\qquad{}+(\alpha-\alpha_h)\|\nabla\cdot\vu_{\calL}\|
  \bigl[\|\nabla\cdot(\vu_h-\vu)\|+\|\nabla \cdot\vu\|\bigr]\\
  &\le C\alpha_h h^2\|\ell\|\,\|\calL\|
   +C(\alpha-\alpha_h)(L\alpha^{-1}\|\ell\|)
    \bigl(h\|\calL\|\bigr)\\
  &\qquad{}+C(\alpha-\alpha_h)(L\alpha^{-1}\|\calL\|)\bigl[
  h\|\ell\|+\alpha^{-1}L\|\ell\|\bigr]\\
  &\le C\Big(\alpha_h h^2+(\alpha-\alpha_h)L\alpha^{-1}[
h+\alpha^{-1}L]\Big)\|\calL\|\|\ell\|.
\end{align*}
Using $(\alpha-\alpha_h)L\alpha^{-1}=\alpha Lh/(L+\alpha h)$, $h+\alpha^{-1}L=(\alpha h+ L)/\alpha$ and $\alpha_h h\le L$, it follows that
\[
|\calL(\vu)-\calL(\vu_h)|
  \le C(\alpha_h h^2 +Lh)\|\calL\|\|\ell\|\le CLh\|\calL\|\|\ell\|,
\]
which completes the proof of \eqref{convergence calL}.
\end{proof}

\begin{corollary}\label{cor: L2 H1 errors}
For $\vu_h$ and $\vu$ as above,
\[
\|\vu_h-\vu\|\le C\mu^{-1}Lh \|\vf\|
\quad\text{and}\quad\|\vu_h-\vu\|_{\vV}\le  C\lambda_h^{1/2} \mu^{-3/2} h \|\vf\|.
\]
\end{corollary}
\begin{proof}
By choosing 
$\calL(\vv)=\int_\Omega(\vu_h-\vu)\cdot \vv\,\ud\bsx$
in~\eqref{convergence calL}, we have
\[
\|\vu_h-\vu\|^2=\calL(\vu_h-\vu)=\calL(\vu_h)-\calL(\vu)\le
 C L h\|\ell\|\|\vu_h-\vu\|\le C\mu^{-1}L h\|\vf\|\|\vu_h-\vu\|,
\]
implying the first estimate. By \eqref{eq: coer B}~and \eqref{h1 error},
\[
\|\vu_h-\vu\|_{\vV}^2\le C\|\varepsilon(\vu_h-\vu)\|^2
\le C\|\vu_h-\vu\|_{\calB_h}^2\le  C \alpha_h h^2\|\ell\|^2\le  C \lambda_h \mu^{-3} h^2\|\vf\|^2,
\]
implying the second estimate.
\end{proof}

\begin{remark}\label{rem: compare}
If we compute $\vu_h$ using the unmodified bilinear form $\calB$, or in other words, if $\vu_h$ is the usual conforming finite element solution, then the arguments above simplify to yield the bounds
\[
\|\vu_h-\vu\|_{\calB}\le\|\vu-\Pi_h\vu\|_{\calB}\le C\alpha^{1/2}h\|\ell\|
=C\mu^{-3/2}\lambda^{1/2}h\|\vf\|
\]
with
\[
|\calL(\vu_h)-\calL(\vu)|=|\calB(\vu-\vu_h,\mathbf{\Theta}_h-\vu_{\calL})|
  \le C \alpha h^2\|\calL\|\|\ell\|= C \mu^{-2}\lambda h^2\|\calL\|\|\vf\|
\]
and
\[
\|\vu_h-\vu\|\le C \mu^{-2}\lambda h^2\|\vf\|.
\]
Since $C_1\mu^{-1}Lh<C_2\mu^{-2} \lambda h^2$ if and only if $\lambda>(C_1/C_2)\mu L/h$, the 
modification based on~\eqref{eq: alpha_h} should result in a smaller $L^2$-error once $\lambda$ is 
larger than a modest multiple of~$\mu L/h$.  Notice that $\lambda_h=\alpha_h\mu\to \mu L/h$ 
as~$\lambda\to\infty$.
\end{remark}
\begin{remark}
To discuss mixed boundary conditions for the $2D$ case, let $\boldsymbol{n}$ denote the outward 
unit normal vector with respect to $\Omega$. Recall that $\Omega$ is a convex polygonal domain, and let the 
closed line segment $\Gamma_i$ denote the $i$th edge of the boundary $\Gamma=\partial\Omega$, ordered 
counterclockwise.  Thus, $\Gamma=\bigcup_{i=1}^m \Gamma_i$ for some finite $m$. We assume  that 
$u={\bf 0}$ on $\Gamma_D=\bigcup_{i=1}^{m_0} \Gamma_i$ and that $\vsigma\boldsymbol{n}={\bf g}$ 
on $\Gamma_N=\bigcup_{m_0+1}^{m} \Gamma_i$.  In other words, homogeneous Dirichlet boundary conditions
apply on $\Gamma_D$, whereas traction boundary conditions apply on $\Gamma_N$. The measure of $\Gamma_D$
must be positive, and the Neumann data ${\bf g}\in{\bf L}^2(\Gamma_N)$.

For $\ell(\vv)=\int_\Omega(\mu^{-1}\vf\cdot\vv)\,\ud\bsx+\int_{\Gamma_N}({\bf g}\cdot\vv)\,\ud s$, the 
existence and uniqueness theorem (Theorem~\ref{thm: unique solution}) remains valid with 
$|\ell(\vv)|\le C(\mu^{-1}\|\vf\|_{\vV^*}+\|{\bf g}\|_{{\bf L}^2(\Gamma_N)})\|\vv\|_{\vV}$ and
$\vV=\{\,\vw\in [H^1(\Omega)]^2:\text{$\vw=0$ on $\Gamma_D$}\,\}$. The conforming finite element scheme \eqref{FEM new} is also well-defined if we put
\[\vV_h=\{\,\vw_h\in\vV:\text{$\vw_h|_K\in[P_1]^2$ for all $K\in\mathcal{T}_h$}\,\}.\]
The convergence results in Theorem \ref{Convergence theorem} hold true provided the regularity results 
in Theorem~\ref{thm: vu bound} remain valid. In these two theorems, we anticipate that 
 $\mu^{-1}\|\vf\|$ should be replaced with  
$\mu^{-1}\|\vf\|+\|{\bf g}\|_{[H^{1/2}(\Gamma_N)]^2}$.  For more details, see for example Arnold and 
Falk~\cite{ArnoldFalk1987}, Brenner and Sung~\cite[Lemma 2.3]{BrennerSung1992}, and Brenner and Scott~\cite[Page 319]{BrennerScott2008}.
\end{remark}

\begin{remark} 
Many natural materials appear macroscopically homogeneous but are actually heterogeneous at the microscopic level, consisting of various constituents or phases. 
In this remark, we discuss the case of inhomogenous elastic material that is nearly incompressible. We define the variable Lam\'e parameters  $\mu$  and $\lambda$  in terms of the Young’s modulus $E$ via 
\[\mu(\bsx)=\frac{E(\bsx)}{2(1+\nu)}~~{\rm and}~~ \lambda(\bsx)=  \Lambda \widehat \lambda(\bsx),\quad {\rm with}~~\Lambda =\frac{1}{1-2\nu}~~{\rm and}~~\widehat \lambda (\bsx)=\frac{E(\bsx)\nu}{1+\nu}.
\]
The constant factor $\Lambda$ is expected to be very large (and approaches  infinity for materials which are increasingly incompressible as the Poisson ratio $\nu$ approaches $1/2$).

For the well-posedness of the problem \eqref{eq:L1}, we assume that $0<E_{\min}\le E\le E_{\max}<\infty$  for some positive constants $E_{\min}$ and $E_{\max}$. The regularity result in Theorem \ref{thm: vu bound} remains valid in this case with $\alpha=\Lambda/\overline\mu$, where $\overline \mu$ is the continuous average of $\mu$ over $\Omega$ and   the constant $C$ depends on $\Omega$, $\mu$ and $\widehat \lambda$, but is independent of $\Lambda$; see  \cite[Theorem 3.2]{Dicketal2024}. We modify the conforming FEM in \eqref{FEM new} as 
\begin{equation*}
\int_\Omega\Big[2\frac{\mu}{\overline \mu}\,\veps(\vu_h):\veps(\vv_h)+\alpha_h
\widehat \lambda  (\nabla\cdot\vu_h)(\nabla\cdot\vv_h)\Big]\,\ud\bsx=\frac{1}{\overline\mu}\int_\Omega\vf \cdot \vv_h \,\ud\bsx,~~ {where}~~\alpha_h=\frac{\Lambda_h}{\overline \mu}=\frac{\alpha}{1+\alpha h/L}.
\end{equation*} 
The achieved convergence results can be extended to cover this case. 
\end{remark}

\section{Numerical experiments}\label{Sec: Numeric}
\begin{table}
\newcommand{\Error}{$\|\vu_h-\vu\|$}
\newcommand{\Rate}{\textrm{Rate}}
\renewcommand{\arraystretch}{1.2}
\caption{Errors and convergence rates for ${\bf u}$ in Example~1.}
\label{tab: L2 errors}
\centering
\texttt{
\begin{tabular}{rc|cc|ccc}
\\\hline
\multicolumn{2}{c|}{$\lambda=100$}&\multicolumn{2}{c|}{$\calB$}&
\multicolumn{3}{c}{$\calB_h$}\\
\hline
$N_h$&$h$&\Error&\Rate&\Error&\Rate&$\lambda_h$\\
\hline
 586& 0.239& 1.07e+00& 1.146& 3.34e-01& 1.296&  15.7 \\
  2466& 0.119& 3.50e-01& 1.612& 1.36e-01& 1.291&  27.1 \\
 10114& 0.060& 9.79e-02& 1.838& 5.46e-02& 1.321&  42.7 \\
 40962& 0.030& 2.56e-02& 1.933& 2.14e-02& 1.353&  59.8 \\
164866& 0.015& 6.52e-03& 1.975& 8.64e-03& 1.307&  74.9 \\
\hline\\\hline
\multicolumn{2}{c|}{$\lambda=1{,}000$}&\multicolumn{2}{c|}{$\calB$}&
\multicolumn{3}{c}{$\calB_h$}\\
\hline
$N_h$&$h$&\Error&\Rate&\Error&\Rate&$\lambda_h$\\
\hline
 586& 0.239& 2.96e+00& 0.283& 3.64e-01& 1.218&  18.3 \\
  2466& 0.119& 1.76e+00& 0.748& 1.64e-01& 1.146&  35.9 \\
 10114& 0.060& 7.02e-01& 1.326& 7.65e-02& 1.102&  69.3 \\
 40962& 0.030& 2.15e-01& 1.706& 3.56e-02& 1.105& 129.6 \\
164866& 0.015& 5.86e-02& 1.877& 1.60e-02& 1.150& 229.4 \\
\hline\\\hline
\multicolumn{2}{c|}{$\lambda=10{,}000$}&
\multicolumn{2}{c|}{$\calB$}&\multicolumn{3}{c}{$\calB_h$}\\
\hline
$N_h$&$h$&\Error&\Rate&\Error&\Rate&$\lambda_h$\\
\hline
 586& 0.239& 3.73e+00& 0.034& 3.67e-01& 1.209&  18.6 \\
  2466& 0.119& 3.42e+00& 0.127& 1.68e-01& 1.127&  37.1 \\
 10114& 0.060& 2.57e+00& 0.409& 8.04e-02& 1.065&  73.9 \\
 40962& 0.030& 1.33e+00& 0.949& 3.92e-02& 1.034& 146.7 \\
164866& 0.015& 4.77e-01& 1.484& 1.93e-02& 1.027& 289.1 \\
\hline\\\hline
\multicolumn{2}{c|}{$\lambda=100{,}000$}&
\multicolumn{2}{c|}{$\calB$}&\multicolumn{3}{c}{$\calB_h$}\\
\hline
$N_h$&$h$&\Error&\Rate&\Error&\Rate&$\lambda_h$\\
\hline
 586& 0.239& 3.84e+00& 0.003& 3.68e-01& 1.208&  18.6 \\
  2466& 0.119& 3.80e+00& 0.014& 1.68e-01& 1.125&  37.2 \\
 10114& 0.060& 3.66e+00& 0.054& 8.08e-02& 1.061&  74.4 \\
 40962& 0.030& 3.20e+00& 0.194& 3.97e-02& 1.026& 148.6 \\
164866& 0.015& 2.16e+00& 0.569& 1.97e-02& 1.011& 296.8 \\
\hline
\end{tabular}
}
\end{table}
\begin{table}
\newcommand{\Error}{$\|\nabla\vu_h-\nabla\vu\|$}
\newcommand{\Rate}{\textrm{Rate}}
\renewcommand{\arraystretch}{1.2}
\caption{Errors and convergence rates for $\nabla \vu$ in Example~1.}
\label{tab: H1 errors}
\centering
\texttt{
\begin{tabular}{rc|cc|ccc}
\\\hline
\multicolumn{2}{c|}{$\lambda=100$}&\multicolumn{2}{c|}{$\calB$}&
\multicolumn{3}{c}{$\calB_h$}\\
\hline
$N_h$&$h$&\Error&\Rate&\Error&\Rate&$\lambda_h$\\
\hline
 586& 0.239& 2.66e+00& 1.063& 1.26e+00& 1.020&  15.7 \\
  2466& 0.119& 1.02e+00& 1.387& 6.18e-01& 1.028&  27.1 \\
 10114& 0.060& 3.83e-01& 1.407& 3.00e-01& 1.044&  42.7 \\
 40962& 0.030& 1.57e-01& 1.291& 1.44e-01& 1.061&  59.8 \\
164866& 0.015& 7.01e-02& 1.161& 6.90e-02& 1.058&  74.9 \\
\hline\\\hline
\multicolumn{2}{c|}{$\lambda=1{,}000$}&\multicolumn{2}{c|}{$\calB$}&
\multicolumn{3}{c}{$\calB_h$}\\
\hline
$N_h$&$h$&\Error&\Rate&\Error&\Rate&$\lambda_h$\\
\hline
 586& 0.239& 6.83e+00& 0.281& 1.30e+00& 0.993&  18.3 \\
  2466& 0.119& 4.12e+00& 0.730& 6.61e-01& 0.980&  35.9 \\
 10114& 0.060& 1.73e+00& 1.253& 3.37e-01& 0.972&  69.3 \\
 40962& 0.030& 5.96e-01& 1.536& 1.70e-01& 0.982& 129.6 \\
164866& 0.015& 1.96e-01& 1.605& 8.43e-02& 1.015& 229.4 \\
\hline\\\hline
\multicolumn{2}{c|}{$\lambda=10{,}000$}&
\multicolumn{2}{c|}{$\calB$}&
\multicolumn{3}{c}{$\calB_h$}\\
\hline
$N_h$&$h$&\Error&\Rate&\Error&\Rate&$\lambda_h$\\
\hline
 586& 0.239& 8.62e+00& 0.034& 1.31e+00& 0.989&  18.6 \\
  2466& 0.119& 7.89e+00& 0.127& 6.67e-01& 0.973&  37.1 \\
 10114& 0.060& 5.95e+00& 0.406& 3.44e-01& 0.956&  73.9 \\
 40962& 0.030& 3.14e+00& 0.921& 1.78e-01& 0.948& 146.7 \\
164866& 0.015& 1.20e+00& 1.389& 9.21e-02& 0.952& 289.1 \\
\hline\\\hline
\multicolumn{2}{c|}{$\lambda=100{,}000$}&
\multicolumn{2}{c|}{$\calB$}&
\multicolumn{3}{c}{$\calB_h$}\\
\hline
$N_h$&$h$&\Error&\Rate&\Error&\Rate&$\lambda_h$\\
\hline
 586& 0.239& 8.86e+00& 0.003& 1.31e+00& 0.989&  18.6 \\
  2466& 0.119& 8.77e+00& 0.014& 6.67e-01& 0.972&  37.2 \\
 10114& 0.060& 8.45e+00& 0.054& 3.44e-01& 0.954&  74.4 \\
 40962& 0.030& 7.39e+00& 0.194& 1.79e-01& 0.944& 148.6 \\
164866& 0.015& 5.01e+00& 0.561& 9.31e-02& 0.943& 296.8 \\
\hline
\end{tabular}
}
\end{table}

In this section, we illustrate the theoretical convergence results. Based on Corollary \ref{cor: L2 H1 errors}, we expect to observe $O(h)$ and $O(\lambda_h^{1/2} h)$ convergence rates in the ${\bf L}^2(\Omega)$-norm for ${\bf u}$ and $\nabla\mathbf{u}$, respectively, uniformly in $\lambda$, provided that all the imposed assumptions are  satisfied. The Julia code used in these experiments is available on GitHub~\cite{McLean2024}.

\paragraph{Example 1.} For $\Omega=(0,\pi)^2$, we chose the body force so
that the exact solution is
\[
\vu(\boldsymbol{x})=\begin{bmatrix}  u_1(x_1,x_2)\\  u_2(x_1,x_2)\end{bmatrix}
=\begin{bmatrix}
\bigl(\cos(2x_1)-1\bigr)\sin(2x_2)\\ \bigl(1-\cos(2x_2)\bigr)\sin(2x_1)
\end{bmatrix}
+\frac{\sin(x_1)\sin(x_2)}{\lambda}\begin{bmatrix}1\\ 1\end{bmatrix},
\quad\bsx=(x_1,x_2).
\]
Notice $\vu=\mathbf{0}$ on $\partial\Omega$.  We used a family of unstructured triangulations obtained by 
uniform refinement of an initial coarse mesh with $h=0.478$. Table \ref{tab: L2 errors} compares the 
errors and convergence rates in the ${\bf L^2}(\Omega)$-norm of the standard conforming, piecewise-linear 
Galerkin FEM (using the original bilinear form $\calB$) with our modified conforming 
scheme~\eqref{FEM new} (using $\calB_h$ depending on $\lambda_h=\alpha_h\mu$).  In all cases,
$\mu=1$. The number of unknowns in the linear system for each mesh is denoted by $N_h$. As~$\lambda$ 
increases from~$100$ to~$100{,}000$ the material becomes more nearly incompressible and, as expected, 
the standard method exhibits poor convergence. By contrast, our new method achieves $O(h)$ convergence 
uniformly in~$\lambda$.  The only case when use of the unmodified bilinear form~$\calB$ results in a 
(slightly) smaller error is if $\lambda=100$~and $h=0.015$; this is more-or-less as expected since 
$L/h=\sqrt{2}\pi/0.015\approx296\approx3\lambda$; see Remark~\ref{rem: compare}. Table~\ref{tab: H1 errors} compares the errors in~$\nabla{\bf u}$ for the two methods.  Once again, our 
new method is more accurate. In fact, it seems to achieve almost $O(h)$ convergence in this example.
\paragraph{Example 2.} (Cook’s membrane problem \cite{Cook1974}.) This is a benchmark problem for linear elasticity which combines bending and shearing.  As illustrated in Figure~\ref{fig: benchmark problem}, $\Omega$ is the convex polygon formed by connecting the vertices $(0,0)$, $(48,44)$, $(48,60)$, and $(0,44)$. The left-hand side of $\Omega$ is clamped (that is,
$\vu=\mathbf{0}$), and a constant shear load in the vertical direction is applied on the right-hand side, indicated by the red arrow: $\vsigma \vn =(0,g)$ where $g$ is a positive constant.  The remaining  part of the boundary is
traction-free, that is, $\vsigma \vn=\mathbf{0}$, and the body force $\vf=\mathbf{0}$.  The elastic material has Young's
modulus $E=1.12499998125$ MPa and Poisson ratio $\nu=0.499999975$, so  
\[
\lambda=\frac{E\nu}{(1+\nu)(1-2\nu)}=7.5\times10^6\quad\text{and}\quad
\mu= \frac{E}{2(1+\nu)}=0.375.
\]
The shear load gives rise to the functional $\ell(\vv) = \int_{44}^{60}g v_2\,\ud x_2$, 
where $\vv=(v_1,v_2)^T$, $\bsx=(x_1,x_2)$ and $g=1/16$.  Starting from a common initial 
mesh with $h=2.579$, we solved for~$\vu_h$ using both the standard bilinear form~$\calB$ 
(using $\lambda$) and the modified form $\calB_h$ (using $\lambda_h=11.173$). Each node 
was then displaced by the respective solution~$\vu_h$ to obtain the deformed meshes shown 
in Figure~\ref{fig: deformed mesh}.  The locking effect is apparent. We also compared the 
computed values for the vertical displacement at the midpoint~$A=(48,50)$ of the 
right-hand edge with the benchmark value $u_2(A)=16.442$~\cite[p.~3491]{LiuWang2022}.  
Figure~\ref{fig: benchmark} shows that using $\calB$ gives a much less accurate result 
than using $\calB_h$.
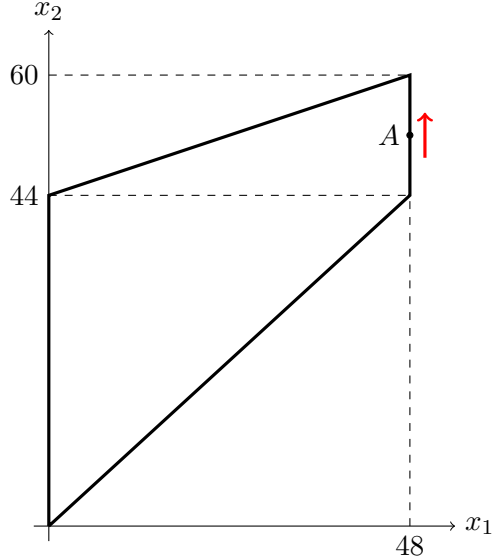
\begin{figure}
\begin{center}
\begin{tikzpicture}[scale=0.07]
\draw[-,very thick] (0,0) -- (48,44) -- (48,60) -- (0,44) -- (0,0);
\node[left] at (0,44) {$44$};
\node[left] at (0,60) {$60$};
\node[below] at (48,0) {$48$};
\draw[->] (-2,0) -- (54,0);
\node[right] at (54,0) {$x_1$};
\draw[->] (0,-2) -- (0,66);
\node[above] at (0,66) {$x_2$};
\draw[dashed] (0,44) -- (48,44);
\draw[dashed] (0,60) -- (48,60);
\draw[dashed] (48,0) -- (48,44);
\node[left] at (48,52) {$A$};
\draw[fill](48,52) circle (0.4);
\draw[->,very thick,red] (50,49) -- (50,55);
\end{tikzpicture}
\end{center}
\caption{The domain $\Omega$ for Cook's membrane problem.}\label{fig: benchmark problem}
\end{figure}

\begin{figure}[!htb]
\centering
\includegraphics[scale=0.5]{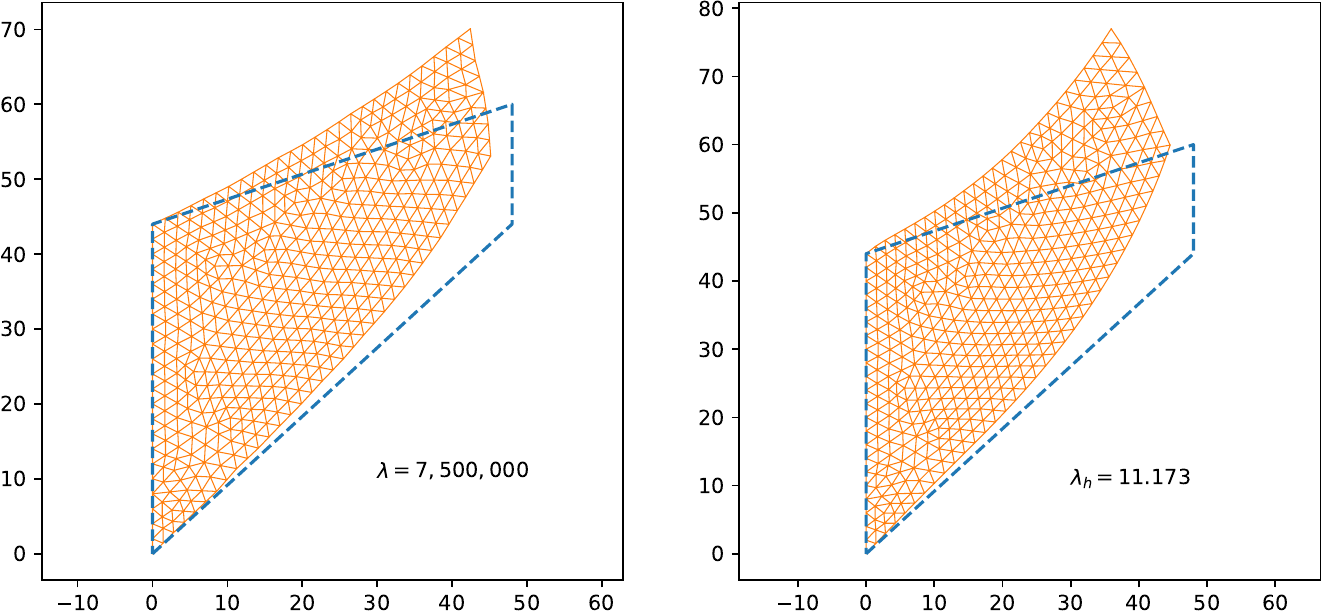}
\caption{The deformed mesh obtained when each node is displaced by $\vu_h$.
Left: using the standard conforming method.  Right: using the modified
method~\eqref{FEM new}.}\label{fig: deformed mesh}
\end{figure}
\begin{figure}[!htb]
\centering
\includegraphics[scale=0.5]{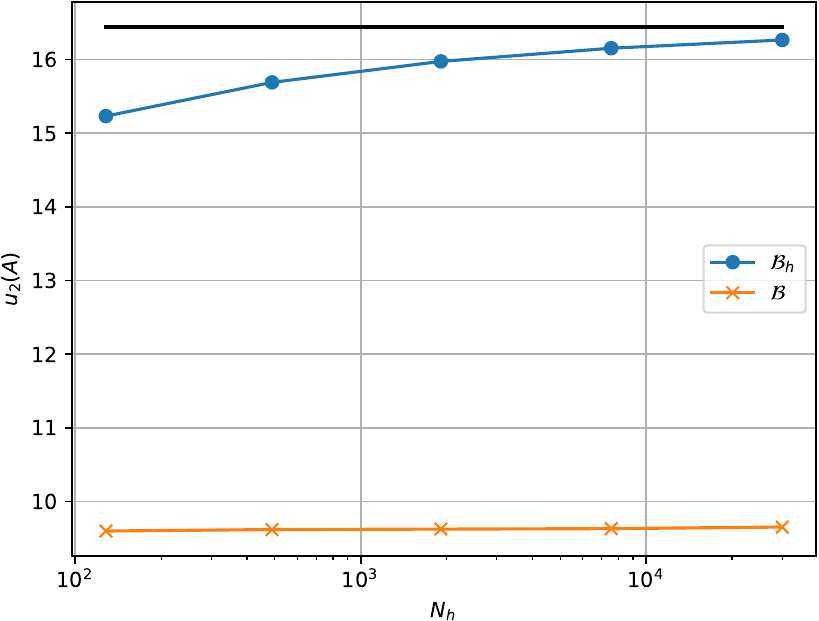}
\caption{Convergence of the computed value of the vertical displacement $u_2(A)$ for each choice of the bilinear form.}
\label{fig: benchmark}
\end{figure}

\end{document}